\documentclass[12pt]{amsart}

\usepackage{amsfonts}
\usepackage{amsmath}
\usepackage{amsthm}
\usepackage{url}
\usepackage[all]{xy}
\usepackage{graphicx}
\usepackage{latexsym}
\usepackage{amssymb}
\usepackage[cp850]{inputenc}
\usepackage{epsfig}
\usepackage{psfrag}
\usepackage{amsfonts}
\usepackage{mathabx}
\usepackage{dsfont}

\newtheorem{sat}{Theorem}[section]		\newtheorem{lem}[sat]{Lemma}
			\newtheorem{prop}[sat]{Proposition}
				
\newtheorem*{defi*}{Definition}			\newtheorem*{bei*}{Example}
\newtheorem*{sat*}{Theorem}				\newtheorem*{kor*}{Corollary}
\newtheorem*{rmk*}{Remark}					

\let\ssection=\section
\renewcommand{\section}{\setcounter{equation}{0}\ssection}

\newtheorem*{namedtheorem}{\theoremname}
\newcommand{\theoremname}{testing}
\newenvironment{named}[1]{\renewcommand{\theoremname}{#1}\begin{namedtheorem}}{\end{namedtheorem}}

\theoremstyle{remark}
\newtheorem*{bem}{Remark}

\newcommand{\BC}{\mathbb C}			
\newcommand{\BR}{\mathbb R}			
			
\newcommand{\BS}{\mathbb S}			\newcommand{\BZ}{\mathbb Z}

		\newcommand{\CL}{\mathcal L}

\newcommand{\actson}{\curvearrowright}

\DeclareMathOperator{\SL}{SL}		
\DeclareMathOperator{\GL}{GL}		
\DeclareMathOperator{\Id}{Id}		
\DeclareMathOperator{\Hom}{Hom}		

\DeclareMathOperator{\SU}{SU}

\newcommand{\comment}[1]{}

\DeclareMathOperator{\SO}{SO}

\DeclareMathOperator{\Stab}{Stab}
\DeclareMathOperator{\U}{U}
\DeclareMathOperator{\Span}{Span}

\DeclareMathOperator{\Her}{Her}

\begin{document}

\title[]{A remark on the homotopy equivalence $\SU_n\simeq\SL_n\BC$}
\author{Juan Souto}
\thanks{The author has been partially supported by NSF grant DMS-0706878, NSF CAREER award 0952106 and the Alfred P. Sloan Foundation}

\begin{abstract}
Being a maximal compact subgroup of $\SL_n\BC$, $\SU_n$ is a deformation retract of the former group. In this note we prove that, for sufficiently large $n$, there is no retraction of $\SL_n\BC$ to $\SU_n$ which preserves commutativity.
\end{abstract}
\maketitle

\section{Introduction}

Using for example the polar decomposition of matrices in $\SL_n\BC$, it is easy to see that $\SU_n$ is a strong deformation retract of $\SL_n\BC$. In fact, there are quite a few ways to prove this result. For instance, it also follows from the Gramm-Schmidt orthogonalization process or from the fact that the symmetric space $\SL_n\BC/\SU_n$ is contractible. Finally, it is a special case of the theorem asserting that every semisimple Lie group retracts to any of its maximal compact subgroups \cite{Helgason}.

The goal of this note is to observe that there is no retraction of $\SL_n\BC$ to $\SU_n$ which, even so mildly, preserves the group structure. More concretely we show:

\begin{sat}\label{no-retract}
For $n\ge 8$, there is no retraction of $\SL_n\BC$ to $\SU_n$ preserving commutativity.
\end{sat}

Before sketching the proof of Theorem \ref{no-retract} we explain briefly what motivated the author to consider the existence or non-existence of commutativity preserving retractions of $\SL_n\BC$ to $\SU_n$. In \cite{patat}, Pettet and the author of this note considered the relation between the representation varieties of $\BZ^k$ in $\SU_n$ and $\SL_n\BC$ and proved that the standard inclusion
$$\Hom(\BZ^k,\SU_n)\hookrightarrow\Hom(\BZ^k,\SL_n\BC)$$
is a homotopy equivalence. If a commutativity preserving retraction of $\SL_n\BC$ to $\SU_n$ existed, then the main result in \cite{patat} would follow almost directly. Given that the proof in \cite{patat} is rather involved, the author was relieved by the conclusion of Theorem \ref{no-retract}.

Also notice that there are highly non-trivial homotopies of $\SL_n\BC$ which preserve commutativity. For instance, there is such a homotopy which retracts $\SL_n\BC$ into the set of diagonalizable matrices \cite{patat}. From this fact, it is easy to deduce (see \cite{patat}) that there is in fact a retraction of $\SL_2\BC$ to $\SU_2$ which preserves commutativity. This shows that some restriction on the $n$ in Theorem \ref{no-retract} is necessary. The bound $n\ge 8$ is due to the rather unsophisticated arguments in our proof. We would however expect that the obvious generalization of Theorem \ref{no-homotopy} remains true for most algebraic groups $G$. For instance, the author expects the following question to have a negative answer:
\medskip

\noindent{\bf Question.} {\em Is there, for $n\ge 4$, a retraction of $\SO_\BR(n,1)_0$ to $\SO_\BR(n)$ which preserves commutativity?}
\medskip

We sketch now the strategy of the proof of Theorem \ref{no-retract}. Arguing by contradiction, assume that there is a continuous homotopy
\begin{equation}\label{eq:homotopy}
\phi:[0,1]\times\SL_n\BC\to\SL_n(\BC),\ \ (t,A)\mapsto\phi_t(A)
\end{equation}
satisfying:
\begin{itemize}
\item[(A1)] If $A,B\in\SL_n\BC$ commute, then so do $\phi_t(A)$ and $\phi_t(B)$ for all $t$.
\item[(A2)] $\phi_0(A)=A$ and $\phi_1(A)\in\SU_n$ for all $A\in\SL_n\BC$.
\end{itemize}
Our first goal is to show that the maps $\phi_t$ fix the center 
$$Z_{\SL_n\BC}(\SL_n\BC)=Z_{\SU_n}(\SU_n)$$
of $\SL_n\BC$. In particular, Theorem \ref{no-retract} follows when we prove:

\begin{prop}\label{no-homotopy}
If $n\ge 8$, then there is no homotopy equivalence 
$$\phi:\SL_n\BC\to\SU_n$$ 
preserving commutativity and with $\phi(Z_{\SL_n\BC}(\SL_n\BC))=Z_{\SU_n}(\SU_n)$.
\end{prop}

The basic idea of the proof of Proposition \ref{no-homotopy} is to associate to any such $\phi$ and any hermitian form $Q$ on $\BC^n$ a linear map $\CL_Q:\BC^n\to\BC^n$ with the property that if $A\in\SU(Q)$ is an isometry of $(\BC^n,Q)$ and $v\in\BC^n$ is an eigenvector of $A$, then $\CL_Q(v)$ is also an eigenvector of $\phi(A)$. This condition imposes rather strong relations between the linear maps $\CL_Q$ and $\CL_{Q'}$ where $Q$ and $Q'$ are different hermitian forms. In fact, we end the proof of Proposition \ref{no-homotopy} showing that these relations cannot be satisfied. 

We derive the existence of the linear map $\CL_Q:\BC^n\to\BC^n$ applying the Fundamental Theorem of Projective Geometry to a suitably constructed map
$$\Phi_Q:P\BC^n\to P\BC^n$$
The construction of $\Phi_Q$ and the proof that it is a colinearity, follow from considerations on the homological properties of the centralizers of certain subgroups of $\SL_n\BC$; essentially we use the fact that complicated elements in $\SL_n\BC$ have much smaller centralizers than simple ones.

\begin{bem}
Many of the steps in the proof of Theorem \ref{no-retract} could be formulated in the more abstract language of algebraic groups. However, for the sake of concreteness, we have chosen to take a rather down-to-earth approach. Altogether, we will only need standard facts of linear algebra and algebraic topology. Not knowing any preferable linear algebra reference work, we refer blankly to Bourbaki \cite{Bourbaki}. Hatcher's book \cite{Hatcher} contains, by far, all the algebraic topology needed below.
\end{bem}

\noindent{\bf Acknowledgements.} The author wishes to thank Alejandro \'Adem, Mladen Bestvina, Rosario Clement, Jos\'e Manuel G\'omez, Lars Louder, Rupert McCallum, Gopal Prasad and Ralf Spatzier for their help, interest, remarks, etc... This work would have not been possible without the collaboration of Alexandra Pettet.

\section{Preliminaries}
In this section we remind the reader of a few well-known facts and fix notation used throughout the paper:

\subsection{Linear algebra}
Throughout this note we will work in the ambient space $\BC^n$. We denote by $\Her(\BC^n)$ the cone of hermitian forms on $\BC^n$. Elements in $\Her(\BC^n)$ will be denoted by $Q$; in particular the {\em standard hermitian form}
$$Q_0((u_i),(v_i))=\sum_i\bar u_iv_i$$
will be denoted $Q_0$. Given $Q\in\Her(\BC^n)$ and $V\subset\BC^n$ a linear subspace, let 
$$V^{\perp_Q}=\{u\in\BC^n\vert Q(u,v)=0\ \hbox{for all}\ v\in V\}$$
be the $Q$-orthogonal complement of $V$. For the standard hermitian form we will use the notation $V^\perp=V^{\perp_{Q_0}}$.

The subgroups of $\GL_n\BC$ and $\SL_n\BC$ preserving $Q\in\Her(\BC^n)$ will be denoted by
\begin{align*} 
\U(Q)&=\{A\in\GL_n\BC\vert Q(Au,Av)=Q(u,v)\ \forall u,v\in\BC^n\}\\
\SU(Q)&=\U(Q)\cap\SL_n\BC
\end{align*}
Notice that in particular $\U_n=\U(Q_0)$ and $\SU_n=\SU(Q_0)$.

Given a group $G$ and $X\subset G$ a subset we denote by
$$Z_G(X)=\{g\in G\vert gx=xg\ \hbox{for all}\ x\in X\}$$
the {\em centralizer} of $X$ in $G$. In particular, $Z_G(G)$ is just the {\em center} of the group $G$. At this point we wish to remind the reader that 
$$Z_{\SL_n\BC}(\SL_n\BC)=Z_{\SU_n}(\SU_n)$$ 
is the discrete and disconnected subgroup consisting of homotheties of $\BC^n$ of ratio an $n$-th root of unity.

Under a {\em commuting subset} $X\subset G$ of a group $G$ we understand a subset with the property that $xx'=x'x$ for all $x,x'\in X$. Recall that if $X$ is a commuting subset of $\GL_n\BC$ and every element of $X$ is diagonalizable, then $X$ is {\em simultanously diagonalizable}, meaning that there is some direct sum decomposition $\BC^n=\oplus_i V_i$ of $\BC^n$ such that every $A\in X$ preserves $V_i$ and restricts to a homothety $A\vert_{V_i}$ of $V_i$ for all $i$. Recall also that if we choose the direct sum decomposition $\BC^n=\oplus_i V_i$ in such a way that the factors $V_i$ are the non-trivial intersections of eigenspaces of all elements in $X$, then it is canonical. More precisely, we will use below the following well-known fact:

\begin{lem}\label{dsd}
Let $X\subset\GL_n\BC$ be a commuting subset consisting of diagonalizable elements. Then there is a uniquely determined direct sum decomposition $\BC^n=\oplus_iV_i$ such that
$$Z_{\GL_n(\BC)}(X)=\{A\in\GL_n\BC\vert AV_i=V_i\ \hbox{for all}\ i\}$$
Here, the factors $V_i$ are the non-trivial intersections of the eigenspaces of the elements in $X$. In particular, the restriction of $g\in X$ to any factor $V_i$ is a homothety. \qed
\end{lem}

We will refer to the direct sum decomposition provided by Lemma \ref{dsd} as the {\em diagonalizing direct sum decomposition} of $X$. Before moving on, recall that every element in $\SU_n$ is diagonalizable and that the corresponding eigenspaces are orthogonal to each other.

Often, we will write something like $\SU_r\times\U_s\subset\GL_n\BC$ for $r+s\le n$. By this we mean the inclusion 
\begin{equation}\label{eq-inclusion}
(A,B)\mapsto 
\left(
\begin{array}{ccc}
A  & 0 & 0 \\
0 & B & 0 \\
0& 0  & \Id
\end{array}
\right)
\end{equation}
where $0$ is the matrix, of the appropriate size, whose entries are all $0$, and where $\Id$ is the identity matrix of again the appropriate size.

Finally, if a group $G$ acts on some set $X$ and $X_0\subset X$ is some subset, we denote by $\Stab_G(X_0)$ the {\em stabilizer} of $X_0$ under the action $G\actson X$. Below we will consider the standard actions of $\SL_n\BC$ on $\BC^n$, the set of linear subspaces of $\BC^n$, projective space $P\BC^n$, etc... For $G\subset\SL_n\BC$ and $X_0$ any set of similar objects we will often just write $\Stab_G(X)$ without making explicit mention to the action; we hope that this does not cause any confusion.

\subsection{Topology}
Through out this note we will only consider homology $H_*(\cdot)$ and cohomology $H^*(\cdot)$ with coefficients in $\BZ$; we feel therefore justified to drop any reference to the coefficients from our notation. 

The cohomology groups of $\U_n$ and $\SU_n$ are well-known \cite[p.434]{Hatcher}. The key fact needed in this note is that inclusions between these groups induce surjections in cohomology. More precisely, for any $k<n$ all the arrows in the following diagram, the pull-backs of the standard inclusions, are surjective:
$$\xymatrix{
H^*(\U_n)\ar[d]\ar[r] & H^*(\SU_n)\ar[d]\\
H^*(\U_k)\ar[r] & H^*(\SU_k)
}$$
We remind the reader that we have homotopy equivalences $\SL_n\BC\simeq\SU_n$ and $\GL_n\BC\simeq\U_n$. In particular, the surjectivity of $H^*(\GL_n\BC)\to H^*(\SU_k)$ for $k<n$, implies:

\begin{lem}\label{lem:class}
The image under the standard inclusion $\SU_k\hookrightarrow\GL_n\BC$ of the fundamental class of $[\SU_k]$ is non-trivial in $H_{k^2-1}(\GL_n\BC;\BZ)$.\qed
\end{lem}

Many of the arguments below will involve easy estimates on the dimension of certain subgroups. We will often use, without any further mention, that $\U_n$ and $\SU_n$ are closed connected manifolds of (real) dimension
$$\dim\U_n=n^2\ \ \hbox{and}\ \ \dim\SU_n=n^2-1$$
In particular, $\SL_n\BC$ and $\GL_n\BC$ are also connected.

\section{Reducing to Proposition \ref{no-homotopy}}
In this section we reduce the claim of Theorem \ref{no-retract} to prove Proposition \ref{no-homotopy}. Our first goal is to prove that elements in $\GL_n\BC$ whose centralizer is sufficiently large (from a homological point of view) belong to the center:

\begin{lem}\label{big-center}
If $A\in\GL_n\BC$ is such that $Z_{\GL_n\BC}(A)$ carries a non-trivial class in $H_{n^2-1}(\GL_n\BC)$, then $A$ belongs to the center of $\GL_n\BC$.
\end{lem}

\begin{proof}
To begin with, consider the multiplicative Jordan decomposition $A=DU$ of $A$. Recall that $D$ is diagonalizable, $U$ is unipotent and if $B$ commutes with $A$ then $B$ also commutes with both $D$ and $U$. By the commutativity properties of the Jordan decomposition we have
\begin{equation}\label{eq101}
Z_{\GL_n\BC}(A)=Z_{\GL_n\BC}(D)\cap Z_{\GL_n\BC}(U)
\end{equation}
We claim that $U=\Id$ and $D\in Z_{\GL_n\BC}(\GL_\BC)$.

In order to prove that $D\in Z_{\GL_n\BC}(\GL_\BC)$ it suffices to show that $D$ has a single eigenvalue. Let $\lambda_1,\dots,\lambda_r$ be the eigenvalues of $D$ and $E_1,\dots,E_r$ the corresponding eigenspaces. We denote by $d_i=\dim E_i$ the multiplicity of $\lambda_i$. By Lemma \ref{dsd}, the centralizer of $D$ is exactly the group of those elements in $\GL_n\BC$ preserving the eigenspace $E_i$ for each $i$. In particular, $Z_{\GL_n\BC}(D)$ is conjugated within $\GL_n\BC$ to the subgroup $\GL_{d_1}\BC\times\dots\times\GL_{d_r}\BC$. It follows that any maximal compact subgroup of $Z_{\GL_n\BC}(D)$ is conjugated within $\GL_n\BC$ to $\U_{d_1}\times\dots\times\U_{d_r}$. Since $Z_{\GL_n\BC}(D)$ is homotopy equivalent to any of its maximal compact subgroups we deduce that
$$H_d(Z_{\GL_n\BC}(D))=0\ \ \hbox{for}\ d>d_1^2+\dots+d_r^2$$
On the other hand $Z_{\GL_n\BC}(A)\subset Z_{\GL_n\BC}(D)$ carries, by assumption, a non-trivial class in $H_{n^2-1}(\GL_n\BC)$. We deduce that
$$n^2-1\le d_1^2+\dots+d_r^2$$
Taking into account that $d_1,\dots,d_r$ are positive integers with $\sum_i d_i=n$, it follows that $r=1$ and $d_1=n$. We have proved that $D$ has a single eigenvalue and hence that $D\in Z_{\GL_n\BC}(\GL_n\BC)$.

We claim now that the unipotent part $U$ of $A$ is trivial. To see that this is the case, let $E\subset\BC^n$ be the eigenspace of $U$ to the eigenvalue $1$ and notice that, since $U$ is unipotent, $E\neq 0$. Again, the centralizer $Z_{\GL_n\BC}(U)$ of $U$ stabilizes $E$. Hence, $Z_{\GL_n\BC}(U)$ is conjugated into the subgroup $G$ of $\GL_n\BC$ of matrices of the following form
$$\left(
\begin{array}{ccc}
C_{d_1,d_1}  & C_{d_1,d_2}  \\
0  & C_{d_2,d_2}
\end{array}
\right)$$
Here, the subscripts represent the size of each block and $d_1=\dim E$. Consider also $G'=\GL_{d_1}\BC\times\GL_{d_2}\BC$ the group of matrices of the following form:
$$\left(
\begin{array}{ccc}
C_{d_1,d_1}  & 0  \\
0  & C_{d_2,d_2}
\end{array}
\right)$$
The obvious projection $G\to G'$ is a fibration with fibers homeomorphic to $\BC^{d_1d_2}$. Hence $G$ and $G'$ are homotopy equivalent. The same argument as above shows first that $H_d(G,\BR)=0$ for all $d>d_1^2+d_2^2$, and then that $d_1=n$. This shows that $E=\BC^n$ and hence that $U=\Id$, as we wanted to show. This concludes the proof of Lemma \ref{big-center}.
\end{proof}

We are now ready to prove that  any homotopy 
$$\phi:[0,1]\times\SL_n\BC\to\SL_n(\BC),\ \ (t,A)\mapsto\phi_t(A)$$
satisfying conditions (A1) and (A2) from the introduction has the property that $\phi_t$ fixes pointwise the center of $\SL_n\BC$.

\begin{lem}\label{center}
$\phi_t$ fixes $Z_{\SL_n\BC}(\SL_n\BC)$ pointwise for all $t\in[0,1]$.
\end{lem}
\begin{proof}
Recall that the center of $\SL_n\BC$ is discrete. Therefore, it suffices to show that for each $A\in Z_{\SL_n\BC}(\SL_n\BC)$ we have $\phi_t(A)\in Z_{\SL_n\BC}(\SL_n\BC)$ for all $t$. For any such $A$ and $t$ we have
\begin{equation}\label{eq-center1}
\phi_t(\SU_n)\subset \phi_t(\SL_n\BC)=\phi_t( Z_{\SL_n\BC}(A))\subset Z_{\SL_n\BC}(\phi_t(A))
\end{equation}
where the final inclusion holds because $\phi_t$ preserves commutativity. Since $\phi_t$ is a homotopy starting with the identity, it follows from Lemma \ref{lem:class}, that
\begin{equation}\label{eq-center2}
[\phi_t(\SU_n)]=[\SU_n]\neq 0\in H_{n^2-1}(\GL_n\BC)
\end{equation}
From \eqref{eq-center1} and \eqref{eq-center2} we deduce that $Z_{\SL_n\BC}(\phi_t(A))$ carries a cycle representing a non-trivial class in $H_{n^2-1}(\GL_n\BC)$. Lemma \ref{big-center} applies and shows that $\phi_t(A)$ is central in $\GL_n\BC$, and hence in $\SL_n\BC$, as we needed to prove. 
\end{proof}

Observe that the final map $\phi_1:\SL_n\BC\to\SU_n$ of the homotopy $\phi_t$ is a homotopy equivalence preserving commutativity and with the property that 
$$\phi_1(Z_{\SL_n\BC}(\SL_n\BC))=Z_{\SL_n\BC}(\SL_n\BC)=Z_{\SU_n}(\SU_n)$$
The content of Proposition \ref{no-homotopy} is that such a homotopy equivalence cannot exist:

\begin{named}{Proposition \ref{no-homotopy}}
If $n\ge 8$, then there is no homotopy equivalence 
$$\phi:\SL_n\BC\to\SU_n$$ 
preserving commutativity and with $\phi(Z_{\SL_n\BC}(\SL_n\BC))=Z_{\SU_n}(\SU_n)$.
\end{named}

In particular, Theorem \ref{no-retract} follows once we have proved Proposition \ref{no-homotopy}; the remaining of this paper is devoted to its proof.

\section{The map $\Phi$}
In this section we associate to any homotopy equivalence $\phi$ as in Proposition \ref{no-homotopy} and to any hermitian form $Q\in\Her(\BC^n)$ a projective transformation of the projective space $P\BC^n$ of $\BC^n$. Recall that $\SU(Q)$ is the subgroup of $\SL_n\BC$ preserving the hermitian form $Q$. As mentioned above, we denote by $Q_0$ the standard hermitian form of $\BC^n$ and hence have that $\SU_n=\SU(Q_0)$.

Before going any further we need some notation. Given a hermitian form $Q\in\Her(\BC^n)$ and a 1-dimensional linear subspace $L\in P\BC^n$ of $\BC^n$ let $\sigma_Q^L\subset\SU(Q)$ be the subgroup of $\SU(Q)$ consisting of elements diagonalized by the direct sum decomposition $\BC^n=L\oplus L^{\perp_Q}$. In more concrete, but also more obscure, terms
$$\sigma_Q^L=\{A\in\SU(Q)\vert\exists\lambda,\mu\in\BC\ \hbox{with}\ Av=\lambda\pi_L(v)+\mu\pi_{L^{\perp_Q}}(v)\ \forall v\in\BC^n\}$$
where
$$\pi_L:\BC^n\to L\ \ \hbox{and}\ \ \pi_{L^{\perp_Q}}:\BC^n\to L^{\perp_Q}$$
are the $Q$-orthogonal projections onto $L$ and onto its $Q$-orthogonal complement $L^{\perp_Q}$. 

\begin{lem}\label{big-center2}
Assume that $n\ge 3$. If a commuting set $X\subset\SU_n$ is not contained in the center $Z_{\SU_n}(\SU_n)$ and has the property that $Z_{\SU_n}(X)$ carries a non-trivial class in $H_{(n-1)^2-1}(\SU_n)$, then there is a unique $L\in P\BC^n$ with $X\subset\sigma_{Q_0}^L$. Furthermore, we have:
$$Z_{\SU_n}(X)=Z_{\SU_n}(\sigma_{Q_0}^L)=\Stab_{\SU_n}(L)$$
\end{lem}
\begin{proof}
Recall that every element in $\SU_n$ is diagonalizable; hence Lemma \ref{dsd} applies to $X$. Let $\BC^n=\oplus_i V_i$ be the diagonalizing direct sum decomposition of $X$ ordered in such a way that $\dim V_i\le\dim V_{i+1}$; notice that the assumption that $X$ is not contained in the center of $\SU_n$ amounts to $r\ge 2$. 

By Lemma \ref{dsd} and the observations after that lemma, the factors $V_i$ are pairwise orthogonal to each other, and
$$Z_{\SU_n}(X)=\{A\in\SU_n\vert AV_i=V_i\ \forall i\}$$
If we set $d_i=\dim V_i$, it follows that $Z_{\SU_n}(X)$ is conjugated within $\SU_n$ to the subgroup $(\U_{d_1}\times\dots\times\U_{d_r})\cap\SU_n$. In particular, $Z_{\SU_n}(X)$ has dimension $d_1^2+\dots+d_r^2-1$. Our homological assumption implies then that
$$d_1^2+\dots+d_r^2-1\ge n^2-1$$
Since $d_1\le\dots\le d_r$ are positive integers with $\sum d_i=n\ge 3$ and since $r\ge 2$, it follows easily that the only possibility is $d_1=1$ and $d_2=n-1$. 

So far we have proved that $Z_{\SU_n}(X)$ fixes the 1-dimensional space $L=V_1$ and hence also its orthogonal complement $L^\perp=V_2$. Notice that, since $X$ is a commuting subset, we have that $X\subset Z_{\SU_n}(X)$. In order to conclude the proof of Lemma \ref{big-center2} we have to show that every $A\in X$ acts on $L^\perp$ as a homothety. In order to see that this is the case, denote by $G$ the subgroup of $\GL(L^\perp)$ preserving the restriction of $Q_0$ to $L^\perp$ and notice that the homomorphism
$$\Stab_{\SU_n}(L^\perp)\to G$$
is in fact an isomorphism. 

Recalling that $Z_{\SU_n}(X)\subset\Stab_{\SU_n}(L^\perp)$ carries, by assumption, a non-trivial class in $H_{(n-1)^2-1}(\SU_n)$, we deduce that the image of $Z_{\SU_n}(X)$ in $G$ carries a non-trivial class in $H_{(n-1)^2-1}(G)$ as well. Since $G$ is isomorphic to $\U_{n-1}$, it follows from Lemma \ref{big-center} that the image of $X$ in $G$ is contained in the center of $G$. In other words, $X$ acts on $L^\perp$ by homotheties, as we wanted to show. 

This concludes the proof of the existence part of Lemma \ref{big-center2}. The uniqueness of $L=V_1$ follows for example from the uniqueness of the diagonalizing direct sum decomposition $\BC^n=\oplus_i V_i$ associated to $X$.

In order to prove the final claim, observe that 
$$\Stab_{\SU_n}(L)=Z_{\SU_n}(\sigma_{Q_0}^L)\subset Z_{\SU_n}(X)$$
because $X\subset\sigma_{Q_0}(L)$. The opposite inclusion follows from the uniqueness of $L$.
\end{proof}

We are now ready to prove the following key fact:

\begin{lem}\label{key1}
If $\phi:\SL_n\BC\to\SU_n$ is a homotopy equivalence as in the statement of Proposition \ref{no-homotopy}, then there is a continuous map
$$\Phi:\Her(\BC^n)\times P\BC^n\to P\BC^n,\ \ (Q,L)\mapsto\Phi_Q(L)$$
such that for $(Q,L)\in\Her(\BC^n)\times P\BC^n$ and $L'\in P\BC^n$ we have
\begin{equation}\label{eq-defi}
L'=\Phi_Q(L)\ \hbox{if and only if}\ \phi(\sigma_Q^L)\subset\sigma_{Q_0}^{L'}
\end{equation}
\end{lem}

Essentially, the statement of Lemma \ref{key1} is that the homotopy equivalence $\phi$ maps subgroups of $\SL_n\BC$ of the form $\sigma_Q^L$ to subgroups of $\SU_n$ of the same form.

\begin{proof}
Given $(Q,L)\in\Her(\BC^n)\times P\BC^n$ notice the the abelian group $\sigma_Q^L\subset\SU(Q)$ is homeomorphic to $\BS^1$ and contains the center of $\SL_n\BC$. In particular, by our assumptions on the homotopy equivalence $\phi$, $\phi(\sigma_Q^L)$ is a connected commuting set containing the center of $\SL_n\BC$. Since the center of $\SL_n\BC$ is disconnected, it follows that 
$$\phi(\sigma_Q^L)\nsubseteq Z_{\SL_n\BC}(\SL_n\BC)$$
Notice now that $\SU_{n-1}$ can be conjugated into the centralizer $Z_{\SL_n\BC}(\sigma_Q^L)$ of $\sigma_Q^L$ and that 
$$\phi(Z_{\SL_n\BC}(\sigma_Q^L))\subset Z_{\SU_n}(\phi(\sigma_{Q_0}^L))$$ 
because $\phi$ preserves commutativity. In particular, it follows from Lemma \ref{lem:class} and the assumption that $\phi$ is a homotopy equivalence that $Z_{\SU_n}(\phi(\sigma_Q^L))$ carries a non-trivial class in $H_{(n-1)^2-1}(\SU_n)$. From Lemma \ref{big-center2} we obtain therefore that there is a unique point in $P\BC^n$, call it $\Phi_Q(L)$, with
$$\phi(\sigma_Q^L)\subset\sigma_{Q_0}^{\Phi_Q(L)}$$
We have established the existence of a map $\Phi:\Her(\BC^n)\times P\BC^n\to P\BC^n$ satisfying 
\eqref{eq-defi}. The continuity of $\Phi$ follows easily from the continuity of $\phi$ and the fact that the former is well-defined.
\end{proof}

At this point we observe a fact that will be used repeatedly below. Given $(Q,L)\in\Her(\BC^n)\times P\BC^n$ recall that when defining $\Phi_Q(L)$ in the proof of Lemma \ref{key1}, we applied Lemma \ref{big-center2} to $X=\phi(\sigma_Q^L)$. The last claim of Lemma \ref{big-center2} implies thus that
$$Z_{\SU_n}(\phi(\sigma_Q^L))=Z_{\SU_n}(\sigma_{Q_0}^{\Phi_Q(L)})$$
Observing that $\Stab_{\SU(Q)}(L)=Z_{\SU(Q)}(\sigma_Q^L)$, we deduce from the commutativity preserving property of $\phi$, that
\begin{align*}
\phi(\Stab_{\SU(Q)}(L))&=\phi(Z_{\SU(Q)}(\sigma_Q^L))\subset Z_{\SU(Q)}(\phi((\sigma_Q^L))\\
&=Z_{\SU_n}(\sigma_{Q_0}^{\Phi_Q(L)})=\Stab_{\SU_n}(\Phi_Q(L))
\end{align*}
This is remarkable enough to be recorded as a lemma:

\begin{lem}\label{key3}
With the same notation as in Lemma \ref{key1} we have 
$$\phi(\Stab_{\SU(Q)}(L))\subset\Stab_{\SU_n}(\Phi_Q(L))$$
for all $(Q,L)\in\Her(\BC^n)\times P\BC^n$.\qed
\end{lem}

Our next goal is to prove that the map $\Phi_Q$ provided by Lemma \ref{key1} is collinear. Recall that a map between projective spaces is {\em collinear} if whenever three points in the domain are contained in some projective line, then their images are also contained in a projective line. 

\begin{lem}\label{key2}
The map $\Phi_Q:P\BC^n\to P\BC^n$ provided by Lemma \ref{key1} is collinear for all $Q\in\Her(\BC^n)$.
\end{lem}
\begin{proof}
Seeking a contradiction, suppose $L_1,L_2$ and $L_3$ in $P\BC^n$ are contained in some projective line but that their images $\Phi_Q(L_1),\Phi_Q(L_2)$ and $\Phi_Q(L_3)$ are not. In other words, if we set 
$$E=\Span(L_1,L_2,L_3)\ \ \hbox{and}\ \ F=\Span(\Phi_Q(L_1),\Phi_Q(L_2),\Phi_Q(L_3))$$
we have $\dim E=2$ and $\dim F=3$. 

Let $G$ be the subgroup of $\SU(Q)$ fixing $E$ pointwise and notice that $G$ stabilizes $E^{\perp_Q}$. Since $G$ fixes $L_i$, we have $\phi(G)\subset\Stab_{\SU(Q)}(\Phi_Q(L_i))$ by Lemma \ref{key3}; in particular,
$$\phi(G)\subset\Stab_{\SU_n}(F)$$
The group $\Stab_{\SU_n}(F)$ is conjugated to the subgroup $(\U_3\times\U_{n-3})\cap\SU_n$ within $\SU_n$. It follows that $\Stab_{\SU_n}(F)$ has dimension $n^2-6n+17$ and hence that
$$H_d(\Stab_{\SU_n}(F))=0\ \ \hbox{for all}\ d>n^2-6n+17$$
On the other hand, $G$ is conjugated in $\SL_n\BC$ to $\SU_{n-2}$. Since $\phi$ is a homotopy equivalence we have by Lemma \ref{lem:class} that $\phi(G)\subset\Stab_{\SU_n}(F)$ represents a non-trivial class in $H_{(n-2)^2-1}(\SU_n)$. This implies that
$$(n-2)^2-1\le n^2-6n+17$$
but this is a contradiction by our assumption that $n\ge 8$. This concludes the proof of Lemma \ref{key2}.
\end{proof}

At this point we want to remind the reader of the following version of the Fundamental Theorem of Projective Geometry:

\begin{named}{Fundamental Theorem of Projective Geometry}
For $n\ge 3$, suppose that $f:P\BC^n\to P\BC^n$ is continuous, colinear, and that $f(P\BC^n)$ contains $n$ points in general position. Then there is a $\BR$-linear map $F:\BC^n\to\BC^n$ with $f(L)=F(L)$ for all $L\in P\BC^n$; moreover, $\det(F)=1$ and $F$ is either $\BC$-linear or $\BC$-antilinear.
\end{named}

\begin{bem}
The author has not found this concrete version of the Fundamental Theorem of Projective Geometry in the literature but assumes it to be well-known. In any case, it follows from Theorem 3.1 in the beautiful Ph.D.-thesis \cite{McCallum} of Rupert McCallum (see also \cite{Lost,McCallum2}) and from the classical result of Kolmogoroff \cite{Kolmogoroff} analyzing the continuous bijective colinearities of projective space.
\end{bem}

We are now ready to prove the main result of this section:

\begin{prop}\label{key}
If $\phi:\SL_n\BC\to\SU_n$ is a homotopy equivalence as in the statement of Proposition \ref{no-homotopy}, then there is a continuous map
$$\Phi:\Her(\BC^n)\times P\BC^n\to P\BC^n,\ \ (Q,L)\mapsto\Phi_Q(L)$$
such that for every $Q\in\Her(\BC^n)$ the following holds:
\begin{enumerate}
\item There is an either linear or antilinear isometry
$$\CL_Q:(\BC^n,Q)\to(\BC^n,Q_0)$$
with $\Phi_Q(L)=\CL_Q(L)$ for all $L\in P\BC^n$.
\item If $A$ is diagonalized by a $Q$-orthogonal basis $(v_1,\dots,v_n)$ then $\phi(A)$ is diagonalized by $(\CL_Q(v_1),\dots,\CL_Q(v_n))$.
\end{enumerate}
\end{prop}
\begin{proof}
Let $\Phi:\Her(\BC^n)\times P\BC^n\to P\BC^n$ be the map provided by Lemma \ref{key1}. By Lemma \ref{key2} we know that for any $Q\in\Her(\BC^n)$ the map
$$P\BC^n\to P\BC^n,\ \ L\mapsto\Phi_Q(L)$$
is colinear. In order to apply the Fundamental Theorem of Projective Geometry we need first to prove that the image of $\Phi_Q$ contains $n$ points in general position. 

In order to see that this is the case let $F\subset\BC^n$ be a minimal linear subspace containing every line in the image of $\Phi_Q$. Given $A\in\SU(Q)$ let $L\in P\BC^n$ be the line generated by some eigenvector of $A$. By Lemma \ref{key3} we have $\phi(A)\in\Stab_{\SU_n}(\Phi_Q(L))$. In other words, we have proved that every element in $\phi(\SU(Q))$ fixes a line contained in $F$. Since the restriction of $\phi$ to $\SU(Q)$ is a homotopy equivalence onto $\SU_n$ we have that $\phi(\SU(Q))=\SU_n$. This implies that $F=\BC^n$, showing that $\Phi_Q(P\BC^n)$ contains $n$ points in general position. 

We can now apply the Fundamental Theorem of Projective Geometry and deduce that for all $Q$ the map $\Phi_Q$ is induced by an either linear or antilinear map 
$$\CL_Q:\BC^n\to\BC^n$$
with $\det(F)=1$. We claim that $\CL_Q$ fulfills (1) and (2).

We start proving (2). Given $A\in\SU(Q)$ diagonalized by $(v_1,\dots,v_n)$ observe first that $(\CL_Q(v_1),\dots,\CL_Q(v_n))$ is a basis of $\BC^n$ because $\CL_Q$ is a, possibly antilinear, isomorphism. It remains to prove that $\CL_Q(v_i)$ is an eigenvector of $\phi(A)$ for each $i$. In order to see that this is the case, let $L_i=\BC v_i$ be the line generated by $v_i$. As above, we obtain
$$\phi(A)\in\Stab_{\SU_n}(\Phi_Q(L_i))=\Stab_{\SU_n}(\CL_Q(L_i))$$
from Lemma \ref{key3}. In other words we have that
$$\phi(A)(\CL_Q(v_i))\in\phi(A)(\CL_Q(L_i))=\CL_Q(L_i)=\BC\CL_Q(v_i)$$
This implies that $\CL_Q(v_i)$ is an eigenvector of $\phi(A)$, as claimed. This concludes the proof of (2). 

It remains to prove that $\CL_Q$ is an isometry of $(\BC^n,Q)$ to $(\BC^n,Q_0)$. We know already that $\det(\CL_Q)=1$. Hence, it suffices to show that $\CL_Q$ maps $Q$-orthogonal lines to $Q_0$-orthogonal lines. Suppose that $L_1,L_2\in P\BC^n$ are $Q$-orthogonal lines; equivalently, $L_1\subset L_2^{\perp_Q}$. This is again equivalent to the conditions 
\begin{itemize}
\item[(C1)] $L_1\neq L_2$, and 
\item[(C2)] $\sigma_Q^{L_1}\subset Z_{\SU(Q)}(\sigma_Q^{L_2})$
\end{itemize}
Since $\CL_Q$ is bijective we have
\begin{itemize}
\item[(C1')] $L_1\neq L_2$.
\end{itemize}
On the other hand we have 
\begin{equation}\label{noideawhat}
\phi(\sigma_Q^{L_1})\subset\phi(Z_{\SU(Q)}(\sigma_Q^{L_2}))=Z_{\SU_n}(\sigma_{Q_0}^{\Phi_Q(L_2)})
\end{equation}
where the inclusion holds because $\phi$ preserves commutativity and the first inclusion holds by the argument used to prove Lemma \ref{key3}. Again by the argument used to prove Lemma \ref{key3}  we have
$$Z_{\SU_n}(\phi(\sigma_Q^{L_1}))=Z_{\SU_n}(\sigma_{Q_0}^{\Phi_Q(L_1)})=Z_{\SU_n}(\sigma_{Q_0}^{\CL_Q(L_1)})$$
This equation and equation \eqref{noideawhat} imply:
\begin{itemize}
\item[(C2')] $\sigma_{Q_0}^{\CL_Q(L_1)}\subset Z_{\SU_n}(\sigma_{Q_0}^{\CL_Q(L_2)})$
\end{itemize}
Conditions (C1') and (C2') are equivalent to the two lines $\CL_Q(L_1)$ and $\CL_Q(L_2)$ being $Q_0$-orthogonal to each other. This proves that $\CL_Q$ is an isometry and hence claim (1).

We have proved Proposition \ref{key}.
\end{proof}

\section{Proof of Proposition \ref{no-homotopy}}
In this section we prove Proposition \ref{no-homotopy}.

\begin{named}{Proposition \ref{no-homotopy}}
If $n\ge 8$, then there is no homotopy equivalence 
$$\phi:\SL_n\BC\to\SU_n$$ 
preserving commutativity and with $\phi(Z_{\SL_n\BC}(\SL_n\BC))=Z_{\SU_n}(\SU_n)$.
\end{named}

We will argue by contradiction, so assume that there is a homotopy equivalence as in Proposition \ref{no-homotopy}. As in the previous section, consider the map
$$\Phi:\Her(\BC^n)\times P\BC^n\to P\BC^n$$
provided by Proposition \ref{key}, or equivalently by Lemma \ref{key1}. Let also 
$$\CL_Q:(\BC^n,Q)\to(\BC^n,Q_0)$$ 
be the isometry provided by Proposition \ref{key} for each $Q\in\Her(\BC^n)$ and recall that $\CL_Q$ is either linear or antilinear. Our first observation is that we may assume that $\CL_Q$ is linear.

\begin{lem}\label{no-anti}
If there is a homotopy equivalence $\phi:\SL_n(\BC)\to\SU_n$ as in Proposition \ref{no-homotopy}, then there is also one such that the map $\CL_Q$ is linear for all $Q\in\Her(\BC^n)$ and such that $\CL_{Q_0}=\Id$.
\end{lem}
\begin{proof}
Notice that $\CL_{Q_0}$ is a, possibly antilinear, isometry of $(\BC^n,Q_0)$ by Proposition \ref{key} (1). In particular, $\CL_{Q_0}^{-1}={ }^t\overline{\CL_{Q_0}}$ and 
$$\SU_n=\CL_{Q_0}\SU_n \CL_{Q_0}^{-1}$$
Considering the map
$$\hat\phi:\SL_n\BC\to \SU_n,\ \ \hat\phi(A)=\CL_{Q_0}\phi(A)\CL_{Q_0}^{-1}$$
we obtain a homotopy equivalence preserving commutativity, fixing the center of $\SL_n\BC$, and with
$$\hat\CL_{Q_0}=\Id$$
where $\hat\CL_Q$ is the map associated to $\hat\phi$ and $Q\in\Her(\BZ^n)$ by Proposition \ref{key}. 

We claim that $\hat\CL_Q$ is linear for all $Q$; equivalently, we have to rule out that it is antilinear. By Proposition \ref{key}, the projective transformation $\hat\Phi_Q$ associated to $\hat\CL_Q$ depends continuously of $Q$. In particular, $H^2(\hat\Phi_Q):H^2(P\BC^n)\to H^2(P\BC^n)$ is the identity for all $Q$ because this is true for $Q_0$ and $\Her(\BC^n)$ is connected. On the other hand, if $F:\BC^n\to\BC^n$ is antilinear and $f:P\BC^n\to P\BC^n$ is the associated projective transformation we have that $H^2(f)=-\Id$. We have proved that $\hat\CL_Q$ is not antilinear, as we needed to show.
\end{proof}

From now on suppose that we have a homotopy equivalence $\phi$ as in Proposition \ref{no-homotopy} with the property that $\CL_Q$ is linear for all $Q$ and that $\CL_{Q_0}=\Id$.

We will derive a contradiction from the following observation:

\begin{lem}\label{compatible}
Suppose that $Q,Q'\in\Her(\BC^n)$ and $L\in P\BC^n$ are such that $L^{\perp_Q}=L^{\perp_{Q'}}$. Then $\CL_Q(L)=\CL_{Q'}(L)$.
\end{lem}
\begin{proof}
The assumption that $L^\perp=L^{\perp_Q}$ implies that
$$\sigma_Q^L=\sigma_{Q'}^L\subset\SU(Q)\cap\SU(Q')$$
From the defining equation \eqref{eq-defi} of $\Phi_Q$ we get that $\Phi_Q(L)=\Phi_{Q'}(L)$, and hence that $\CL_Q(L)=\CL_{Q'}(L)$.
\end{proof}

As a first consequence we deduce that $\CL_Q$ is essentially given by the polar part of any isometry between $(\BC^n,Q)$ and $(\BC^n,Q_0)$. 

\begin{lem}\label{no-homotopy-polar}
If $A:(\BC^n,Q)\to(\BC^n,Q_0)$ is any $\BC$-linear isometry with $\det A=1$ and $P$ is the polar part of $A$, then there is $U\in \SU_n$ with $\CL_Q=UP$. Moreover, $U$ acts as a homothety on every eigenspace of $P$ and hence commutes with $P$.
\end{lem}

Recall that the polar part $P$ of $A$ is the unique positive definite hermitian matrix with ${ }^t\bar PP={ }^t\bar A A$. We have $A=UP$ for some $U\in\SU_n$.

\begin{proof}
To begin with, observe that $P:(\BC^n,Q)\to(\BC^n,Q_0)$ is also an isometry:
$$Q_0(Pv,Pw)={ }^t\bar v{ }^t\bar PP\bar w={ }^t\bar v{ }^t\bar AA\bar w=Q_0(Av,Aw)=Q(v,w)$$
In particular, $U=\CL_QP^{-1}\in\SU_n$. 

We claim that $U$ acts as a homothety on every eigenspace of $P$; recall that $P$ is positive definite hermitian and hence diagonalizable. Moreover, the eigenspaces of $P$ are both $Q_0$-orthogonal and $Q$-orthogonal to each other. This implies that if $L\in P\BC^n$ is contained in an eigenspace of $P$ we have $L^\perp=L^{\perp_Q}$. By Lemma \ref{compatible} and the normalization $\CL_{Q_0}=\Id$ we have thus $\CL_Q(L)=\CL_{Q_0}(L)=L$ and hence
$$U(L)=(\CL_Q(P^{-1}(L))=\CL_Q(L)=L$$
In other words, $U\in\SU_n$ fixes every 1-dimensional subspace of every eigenspace of $P$ and hence acts as a homothety on each one of these eigenspaces. This implies directly that $PU=UP$.
\end{proof}

We are now ready to conclude the proof of Proposition \ref{no-homotopy}. In order to do that consider, for $t\in\BR$, the matrix
$$A_t=
\left(
\begin{array}{ccc}
e^t  & e^t  & 0 \\
0 & e^{-t}  & 0 \\
0 & 0  & \Id_{n-2}
\end{array}
\right)\in\SL_n\BC$$
where $\Id_{n-2}$ is the $(n-2)$-by-$(n-2)$ identity matrix. Consider also the associated family of quadratic forms
$$Q_t(v,w)=Q_0(A_tv,A_tw)$$
and the linear subspaces 
\begin{align*}
&S_1=\BC\times\{0\}\times\dots\times\{0\}\\
&S_2=\{(-x,x,0\dots,0)\vert x\in\BC\}\\
&S_3=\{0\}\times\{0\}\times\BC^{n-2}
\end{align*}
Observe that the images 
\begin{align*}
&A_tS_1=S_1\\
&A_tS_2=\{0\}\times\BC\times\{0\}\times\dots\times\{0\}\\
&A_tS_3=S_3
\end{align*}
of $S_1,S_2,S_3$ under $A_t$ are independent of $t$ and $Q_0$-orthogonal. It follows that the spaces $S_1,S_2$ and $S_3$ are pairwise $Q_t$-orthogonal for all $t$. In particular we have $S_1^{\perp_{Q_t}}=S_2\oplus S_3$ for all $t$. Lemma \ref{compatible} implies that 
\begin{equation}\label{eq-blablabla}
\CL_{Q_t}(S_1)=\CL_{Q_s}(S_1)
\end{equation}
for all $t$ and $s$. We will obtain a contradiction to \eqref{eq-blablabla}. 

Consider the polar part
$$P_t=\sqrt{
\left(
\begin{array}{ccc}
e^{2t}  & e^{2t}  & 0 \\
e^{2t} & e^{2t}+e^{-2t}  & 0 \\
0 & 0  & \Id_{n-2}
\end{array}
\right)}$$
of $A_t$ and notice that it has three distinct eigenvalues $\lambda_1^t\sim e^{t}$, $\lambda_1^t\sim e^{-t}$ and $\lambda_3=1$. The $\lambda_3$-eigenspace is the subspace $S_3$ and hence does not depend on $t$. Denote by $E_t$ the $\lambda_1^t$ eigenspace of $P_t$. A computation shows that 
\begin{align*}
\lim_{t\to\infty}E_t&=\lim_{t\to\infty}P_t(S_1)=\{(x,x,0,\dots,0)\vert x\in\BC\}\\
\lim_{t\to -\infty}E_t&=\lim_{t\to-\infty}P_t(S_1)=S_1
\end{align*}
It follows that for any choice of $U_t\in\SU_n$ preserving the eigenspace $E_t$ of $P_t$ we have
\begin{align*}
\lim_{t\to\infty} U_tP_t(S_1)&=\{(x,x,0,\dots,0)\vert x\in\BC\}\\
\lim_{t\to-\infty} U_tP_t(S_1)&=S_1
\end{align*}
On the other hand, by Lemma \ref{no-homotopy-polar}, there is, for all $t$, some $U_t\in\SU_n$ fixing the eigenspaces of $P_t$ and with $\CL_{Q_t}=U_tP_t$. Hence we have
$$\lim_{t\to\infty}\CL_{Q_t}(S_1)=\{(x,x,0,\dots,0)\vert x\in\BC\}\neq S_1=\lim_{t\to-\infty}\CL_{Q_t}(S_1)$$
contradicting \eqref{eq-blablabla}. This contradiction concludes the proof of Proposition \ref{no-homotopy}.\qed

\begin{bem}
The end of Proposition \ref{no-homotopy} is perhaps a little bit disappointing but we beg the reader to think how to prove that $\GL_2\BR$ is not commutative. The reader would probably just write two random matrices $A$ and $B$, multiply them, and see that $AB\neq BA$. Giving an example to show that something preposterous does not hold is neither conceptual nor sophisticated, but perhaps the only available argument; or at least the simplest.
\end{bem}

\bigskip

\medskip

\noindent{\small Department of Mathematics, University of Michigan, Ann Arbor \newline \noindent
\texttt{jsouto@umich.edu}}

\end{document}